\documentclass[a4paper, 10pt]{amsart}
 \usepackage[foot]{amsaddr}
\usepackage{amsfonts,mathrsfs,amsmath,amscd,amssymb}
\usepackage[pdftex]{graphicx}
\usepackage{hyperref}
\usepackage{color}
\newtheorem{theorem}{Theorem}

\usepackage{enumitem}
\usepackage[english]{babel}
\usepackage[all,cmtip]{xy}

\newtheorem{proposition}[theorem]{\sf Proposition}
\newtheorem{lemma}[theorem]{\sf Lemma}
\newtheorem{definition}[theorem]{\sf Definition}
\newtheorem{corollary}[theorem]{\sf Corollary}
\newtheorem{remark}[theorem]{\sf Remark}

\def \Z{\mathbb Z}

\def \O{\mathcal O}
\def \F{\mathcal F}
\def \T{\mathcal T}
\def \N{\mathcal N}
\def \ra{\rightarrow}

\let\olddefinition\definition
\renewcommand{\definition}{\olddefinition\normalfont}

\let\oldremark\remark
\renewcommand{\remark}{\oldremark\normalfont}

\begin{document}
\author{Tanya Kaushal Srivastava}
\title[Lifting Automorphisms as Autoequivalences]{Lifting Automorphisms on Abelian Varieties as Derived Autoequivalences}
\address{Institute of Science and Technology, 3400 Klosterneuburg, Am Campus 1, Austria}
\email{tanyakaushal.srivastava@ist.ac.at}
\date{\today}
\thanks{I would like to thank Piotr Achinger, Daniel Huybrechts, Katrina Honigs, Marcin Lara and Maciek Zdanowicz for the mathematical discussions and Tamas Hausel for hosting me in his research group at IST Austria. This reserach has been funded by IST plus fellowship.}

\maketitle

\tableofcontents

\begin{abstract}
We show that on an Abelian variety over an algebraically closed field of positive characteristic, the obstruction to lifting an automorphism to an Abelian variety over a field of characteristic zero as a morphism vanishes if and only if it vanishes for lifting it as an autoequivalence. We also compare the deformation space of these two types of deformations.    
\end{abstract}

\maketitle

\smallskip
\noindent \textbf{Keywords.} Abelian Varieties, Automorphisms, Positive characteristic, Derived autoequivalences, Lifting to characteristic zero.\\
\noindent \textbf{2010 Mathematics Subject Classifcation:} 14F05, 14J50, 14G17, 11G10, 14D15, 14K99.

\section*{Introduction}

The main idea on which the result below is based is that when we lift an automorphism of $X$ as an autoequivalence, we basically lift, as a coherent sheaf, the structure sheaf of the graph of the automorphism over the product  $X \times X$, now the support of this lifted sheaf will give a deformation of the graph of the automorphism and since the deformation is flat, it has to come from a lift of the automorphism itself.     

The results  and arguments presented here maybe be well known to experts, but since we did not find them written down in literature, we note them down here.  

We break the paper into 2 sections. In section 1 we recall the basics of Abelian varieties that we would need in for understanding their deformation theory and derived autoequivalences. We compute the necessary Hodge groups and deformation spaces. Moreover, we observe that p-rank of an Abelian variety is a derived invariant. Thus, every Abelian variety derived equivalent to an ordinary Abelian variety is ordinary.  

In section 2, we use the deformation-obstruction sequence for the corresponding deformation functors as morphism, or as autoequivalence and compare the obstruction and deformation space dimensions. 

The main result (Proposition \ref{main2}, Theorem \ref{main1}) can be stated as follows:

\begin{theorem}
An automorphism of an Abelian variety lifts as a morphism if and only if it lifts as a derived autoequivalence. Moreover, in case it lifts, there are extra lifts of the automorphism as an autoequivalence given by composition of the lifted automorphism with twist by the lift of the structure sheaf of the graph.
\end{theorem}

\textbf{Notation:} For an algebraically closed field $k$ of positive characteristic, we will denote by $W(k)$ the ring of Witt vectors over $k$, sometimes also abbreviated as $W$. $K$ will denote a characteristic zero field.

\section{Abelian Varieties}

In this section we discuss basic properties of Abelian varieties paying a special attention to elliptic curves, for which we prove result before the general case. 

\begin{definition}[\cite{MumfordAV}]
An \emph{Abelian variety} $X$ is a complete algebraic variety over an algebraically closed field, $k$, with a group law $m: X \times X \ra X$ such that $m$ and the inverse map are both isomorphisms of varieties. 
\end{definition}

\subsubsection{Hodge Groups  for Abelian Varieties}

For $A$ a $g$-dimensional Abelian variety over $k$, an algebraically closed field, the Hodge groups $H^q(A, \Omega_A^p)$ are canonically isomorphic to $\wedge^p[H^0(A, \Omega_A^1)] \otimes \wedge^q[H^1(A,\O_A)]$, and dim $H^1(A, \O_A) =$ dim $H^0(A, \Omega_A) = g$, see \cite{MumfordAV}.   

\subsubsection{Deformation theory of Abelian Varieties}

\begin{theorem}[Norman-Oort, \cite{NO}]
Every Abelian variety defined over an algebraically closed field of positive characteristic lifts to characteristic zero. 
\end{theorem}

The formal deformation space of the lifting of an Abelian variety is given by $H^1(A, T_A)$, see \cite[Proposition 6.15]{GIT}. This has dimension $g^2$. Since all formal deformations need not be algebraic, the dimension of the algebraic deformations will be smaller for $g \geq 2$.  As in the case of elliptic curves, note that the algebraization of the lift is automatic as the obstruction to lifting an (ample) line bundle lie in the second cohomology group, which vanishes as we are in dimension 1. Thus we can use Grothendieck's existence theorem to conclude.

To construct algebraic lifts of Abelian varieties, we lift a polarization on the base Abelian variety along with it (for a reason why we do this, see remark \ref{subfunctor}) and then use Grothendieck's Existence theorem to conclude. The dimension of the local moduli space of algebraic (polarized) lifts is $g(g+1)/2$, see \cite{OortFG} for details.  We remark that even though every Abelian variety can be lifted to characteristic 0, this is true only in the weak sense, i.e, not every Abelian variety would admit a lift over the Witt ring but one would need a ramified extension of the Witt ring to get a lift, see \cite[Section 13]{OortAVe}. However, ordinary Abelian varieties admit a lift over the Witt ring.  

\subsubsection{The Automorphism Group of Abelian Varieties and Deformations}
We refer to \cite[Chapter IV]{MumfordAV}, for the structure of automorphism group (or more generally endomorphisms) of Abelian varieties. For the explicit description of the group of automorphisms of an elliptic curve see Silverman, \cite[Chapter III, Theorem 10.1]{SilvermanAEC}.  

There exist non-liftable automorphisms for Abelian varieties. Thus , the full automorphism group of an Abelian variety need not lift. For an example  see \cite[Section 14]{OortAVe}. The obstruction and deformation space of a particular automorphism have been computed below in section \ref{lifting}.

\subsection{Derived Equivalent Abelian Varieties}

Let $A$ be an Abelian variety and $\hat{A}$ the dual Abelian variety, we will first recall the definition of the group, $U(A \times \hat{A})$, of isometric automorphisms of $A \times \hat{A}$. 
Note that any morphism $f: A \times \hat{A} \ra A \times \hat{A}$ can be written as a matrix 
$$
\begin{pmatrix} 
a & b \\
c & d 
\end{pmatrix},
$$
where the morphism $a$ maps $A$ to $A$, $b$ maps $\hat{A}$ to $A$ and so on. Each morphism $f$ determines two other morphisms $\hat{f}$ and $\tilde{f}$ from $A \times \hat{A}$ to $A \times \hat{A}$ whose matrices are
$$
\hat{f} = \begin{pmatrix} 
\hat{d} & \hat{b} \\
\hat{c} & \hat{a} 
\end{pmatrix},
$$
and 
$$
\tilde{f} = \begin{pmatrix} 
\hat{d} & - \hat{b} \\
-\hat{c} & \hat{a} 
\end{pmatrix}.
$$

We define $U(A \times \hat{A})$ as the subset of all automorphisms of $A \times \hat{A}$ such that $f^{-1} = \tilde{f}$ and such an automorphism is called {\it isometric}. We can extend this definition to isometric isomorphism of two Abelian varieties $A$ and $B$ and denote the group as $U(A \times \hat{A}, B \times \hat{B})$. The following theorem of Orlov \cite{OrlovAV} and Polishchuk \cite{Pol}, characterizes derived equivalences of two Abelian varieties using isometric isomorphisms.

\begin{theorem} [\cite{HuyFM}, Corollary 9.49] \label{derivedAV}
Two Abelian varieties $A$ and $B$ over an algebraically closed field $k$ define equivalent derived categories  $D^b(A)$ and $D^b(B)$ if and only if there exists an isomorphism $f: A \times \hat{A} \xrightarrow B \times \hat{B}$ with $\tilde{f} = f^{-1}$: 
$$
D^b(A) \cong D^b(B) \iff U(A \times \hat{A}, B \times \hat{B}). 
$$  
\end{theorem}

In the special case of elliptic curves

\begin{theorem}[\cite{AKW} Theorem 2.2]
Let $E$ and $F$ be two elliptic curves over a field $k$ such that $\Phi: D^b(E) \cong D^b(F)$ as a k-linear triangulated categories, then there is an isomorphism of $k$-schemes $E \cong F$.
\end{theorem}

\begin{remark}
Note that the above statement for elliptic curves is valid even over non-algebraically closed fields, we refer the reader to \cite[Corollary 7.4.5]{Liu} for comparison between genus 1 curves and elliptic curves over non-algebraically closed fields. The extension for Abelian Varieties over any field is also expected to hold, but has not been proved yet.   
\end{remark}

Recall that the $p$-rank of an Abelian variety over an algebraically closed field $k$ is the rank $i$ of the $p$-torsion group $A[p] \cong (\Z/p\Z)^i$, where $i \in [0, dim(X)] $, if char $k$ = p otherwise $i = 2g$.    

\begin{corollary}
The $p$-rank of an Abelian variety is a derived invariant. 
\end{corollary}

\begin{proof}
This follows from the easy observation that $\hat{A}[p] = A[p]$ and then using the theorem above as isomorphisms preserve p-torsion.  
\end{proof}

This also follows from the result of Honigs \cite{Hon} that derived equivalent Abelian varieties are isogenous. 

\subsubsection{Derived Autoequivalence Group of Abelian Varieties}

The derived autoequivalence group satisfies the following short exact sequence \cite[Proposition 9.55]{HuyFM}:

\begin{equation}
0 \ra \Z \oplus (A \times \hat{A}) \ra Aut (D^b(A)) \ra U(A \times \hat{A}) \ra 0 
\end{equation}

Here, $A$ is an Abelian variety, $U(A \times \hat{A})$ is the group of isometric automorphisms. Explicitly, the kernel is generated by shifts $[n]$, translations $t_{a*}$, and tensor products $L \otimes ()$ with $L \in Pic^0(A)$.

\section{Lifting automorphisms as autoequivalences} \label{lifting}

Let $E$ be an elliptic curve with j-invariant not equal to $0$ or $1728$ over $k$ a field of characteristic $p > 0$. Let $\sigma: E \ra E$ be an automorphism of $E$. This automorphism induces a Fourier-Mukai equivalence on the derived categories given by the Fourier-Mukai kernel $\O_{\Gamma(\sigma)}$,  where  $\O_{\Gamma(\sigma)}$ is the push forward of the structure sheaf of the graph of $\sigma$ to $X \times X$ and is considered as a coherent sheaf in $D^b(E \times E)$:
$$
\Phi_{\O_{\Gamma(\sigma)}}: D^b(E) \xrightarrow{\cong} D^b(E).  
$$
This reinterpretation of an automorphism as an perfect complex in the derived category provides us with another way of deforming them as a perfect complex, rather than just as deformation as a morphism. However, since we started with deforming a coherent sheaf, the deformations of it as a perfect complex will still be coherent sheaves, see \cite[remark after proof of Theorem 3.4]{tsrivast-1}. More precisely, let $R$ be an Artinian local $W(k)$-algebra with residue field $k$, $E_R$ be an infinitesimal deformation of $E$ over $R$ and consider the following two deformation functors: first is the \textbf{deformation functor of an automorphism as a morphism} 
\begin{equation}
\begin{split}
F_{aut}: & (\text{Artin local $W(k)$-algebras with residue field $k$})  \rightarrow (Sets) \\
& R \mapsto \{\text{Lifts of automorphism $\sigma$ to $R$, i.e., pairs} (E_R, \sigma_R) \},
\end{split}
\end{equation}
where by lifting of automorphism $\sigma$ over $R$ we mean that there exists an infinitesimal deformation $E_R$ of $E$ and an automorphism $\sigma_R: E_R \rightarrow E_R$ which reduces to $\sigma$, i.e., we have the following commutative diagram: 
$$
\xymatrix{
E_R  \ar[r]^{\sigma_R} &E_R\\
E \ar[u] \ar[r]^{\sigma} & E. \ar[u]}
$$
The second one is the \textbf{deformation functor of an automorphism as a coherent sheaf}, defined as follows:
\begin{equation}
\begin{split}
F_{coh}: & (\text{Artin local $W(k)$-algebras with residue field $k$})  \rightarrow (Sets) \\
& R \mapsto \{\text{Deformations of $\mathcal{O}_{\Gamma(\sigma)}$ to $R$} \}/ iso,
\end{split}
\end{equation}
where by deformations of $\mathcal{O}_{\Gamma(\sigma)}$ to $R$ we mean that there exists an infinitesimal deformation $Y_R$ of $Y := E \times E $ over $R$ and a coherent sheaf $\mathcal{F}_R$, which is a deformation of the coherent sheaf $\mathcal{O}_{\Gamma(\sigma)}$ and $\mathcal{O}_{\Gamma(\sigma)}$ is considered as a coherent sheaf on $E \times E$ via the closed embedding $\Gamma(\sigma) \hookrightarrow E \times E$. Isomorphisms are defined in the obvious way. 

Observe that there is a natural transformation $\eta: F_{aut} \rightarrow F_{coh}$ given by 
\begin{equation}
\begin{split}
\eta_R :  F_{aut}(R) & \longrightarrow F_{coh}(R)\\
 (\sigma_R: E_R \rightarrow E_R) & \mapsto \mathcal{O}_{\Gamma(\sigma_R)} / E_R \times E_R.
\end{split}
\end{equation}

There is a deformation-obstruction long exact sequence connecting the two functors.
\begin{proposition}[\cite{tsrivast-1}Proposition 3.6] \label{HartshorneEx19.1}
Let $i: X \hookrightarrow Y$ be a closed embedding with $X$ integral and projective scheme of finite type over $k$. Then there exists a long exact sequence 
\begin{equation} \label{les}
\begin{split}
0 \rightarrow H^0(\mathcal{N}_X) \rightarrow &\text{Ext}^1_Y(\O_X, \O_X) \rightarrow H^1(\O_X) \rightarrow  \\
&H^1(\mathcal{N}_X) \ra \text{Ext}^2_Y(\O_X, \O_X) \ra \ldots,
\end{split}
\end{equation}
where $\mathcal{N}_X$ is the normal bundle of $X$. 
\end{proposition}

\begin{remark}
Note that the obstruction spaces for the functors $F_{aut}$ and $F_{coh}$ are $H^1(\mathcal{N}_X)$ and $\text{Ext}^2_Y(\O_X, \O_X)$ respectively.  See, for example, \cite[Theorem 6.2, Theorem 7.3]{HartshorneDT}. The same results give us the tangent spaces for the functors $F_{aut}$  and $F_{coh}$ and they are  $H^0(\mathcal{N}_X)$ and $\text{Ext}^1_Y(\O_X, \O_X)$.
\end{remark}

\begin{lemma}
The normal bundle of $\Gamma(\sigma): E \ra E \times E$ is trivial.  
\end{lemma}

\begin{proof}
Use \cite[Proposition II.8.20]{HartshorneAG} for the case r = 1. 
\end{proof}

Thus, we have can compute the obstruction and deformation spaces for $F_{aut}$ and they are $T_{F_{aut}} = H^0(E, \mathcal{N}_E) = H^0(E, \O_E) = k$ and $Obs_{F_{aut}} = H^1(E, \mathcal{N}_E) = H^1(E, \O_E) = k$. 

Next, note that there is a canonical isomorphism $Ext^i_{E \times E}(\O_{\Gamma(\sigma)}, \O_{\Gamma(\sigma)})$ and $Ext^i_{E \times E}(\O_E, \O_E)$ given by $\sigma$, so to compute the obstruction and deformation spaces for $F_{coh}$, we need to compute $Ext^i_{E \times E}(\O_E, \O_E)$. Recall that these groups are just the Hochschild cohomology groups of the elliptic curve $E$ and using Hochschild-Kostant-Rosenberg (HKR) isomorphism \cite[Theorem 1.3]{HKR}, we get that
\begin{eqnarray*}
Ext^1_{E \times E}(\O_E, \O_E) & = H^1(E, \O_E) \oplus H^0(E, \mathcal{T}_E^1) \\
& = H^1(E, \O_E) \oplus H^0(E, \O_E)\\
&= k \oplus k,  
\end{eqnarray*}
as $\T_E$ is a free sheaf and 
\begin{eqnarray*}
Ext^2_{E \times E}(\O_E, \O_E) &= H^2(E, \O_E) \oplus H^1(E, \T_E) \oplus H^0(E, \T_E^2)\\
& = H^2(E, \O_E) \oplus H^1(E, \O_E) \oplus H^0(E, \wedge^2 \O_E) \\ 
&= H^1(E, \O_E) = k.
\end{eqnarray*}
 
Thus dimensions of the $Ext$ groups are dim $Ext^1_{E \times E}(\O_E, \O_E) = 2$ and dim $Ext^2_{E \times E}(\O_E, \O_E) = 1$. With this we have computed the deformation and obstruction spaces for $F_{coh}$, that is deformation as a derived autoequivalence, although we still have to show that the lifted structure sheaf of the graph does induce a derived autoequivalence. Indeed, let $\F_W$ be the lift of the structure sheaf of the graph $\O_{\Gamma(\sigma)}$ to $E_W \times E_W$, where $E_W$ is a lift of $E$ over $W$. Then  using Nakayama 's lemma, we note that $\Delta_*\O_{E_W} \ra \F_W \circ \F^{\vee}_W$ and $\F^{\vee}_W \circ \F_W \ra \Delta_*\O_{E_W}$ are quasi-isomorphism, where $\F_W \circ \F^{\vee}_W$ denotes the Fourier-Mukai kernel of the composition $\Phi_{\F_W} \circ \Phi_{\F^{\vee}_W}$, explicitly given by $p_{13*}(p_{12}^*(\F_W) \otimes p^*_{23}(\F_W^{\vee}[1]))$, for the case of elliptic curves. This argument also works for the higher dimensional Abelian varieties case, although the shift in that case is by $[g]$, where $g$ is the dimension of the Abelian variety.

\begin{proposition}
For an elliptic curve $E$ over an algebraically closed field $k$ of characteristic $p > 0$, any automorphism $\sigma: E \ra E$ lifts to a lift $E_K$ of $E$ over a field $K$ of characteristic zero, if and only if the Fourier-Mukai transform induced by the structure sheaf of the graph of $\sigma$, $\O_{\Gamma(\sigma)} \in D^b(E \times E)$, lifts as as an autoequivalence to $D^b(E_K) \ra D^b(E_K)$.   
\end{proposition}

\begin{proof}
From the long exact sequence \ref{les} above, putting in the computations of the groups from the preceding paragraph we get the exact sequence:
$$
0 \ra k \rightarrow k \oplus k \xrightarrow{\alpha} k \xrightarrow{\beta} k \xrightarrow{\gamma} k \ra \ldots, 
$$
where we note that $\alpha$ has to be a surjection (it cannot be zero, as the exactness would then imply that $k \oplus k \cong k$), thereby making $\beta$ a zero map and $\gamma$ an injection.  Thus, we get our result. 
\end{proof}

\subsubsection{The case of higher dimensional Abelian varieties}

In this case, just the computations of obstruction-deformation sequence will be insufficient to conclude that every automorphism lifts as a morphism if and only if it lifts as an autoequivalence. However, we show that is the case by observing that support of the lifted Fourier-Mukai kernel sheaf, is actually the graph of a lift of an automorphism, this argument works for any smooth projective variety, not just only for Abelian varieties, see remark \ref{CY} below. 

Let $A$ be an Abelian variety over an algebraically closed field of positive characteristic and let $\sigma: A \ra A$ be an automorphism of $A$, the definitions of the previous section can be transported directly over to the Abelian varieties from the case of elliptic curves in a straightforward manner.  

\begin{lemma}
The normal bundle $\mathcal{N}_{\Gamma(\sigma)}$ of $\Gamma(\sigma): A \ra A \times A$ is a free sheaf of rank $g$ on $A$, where $g$ is the dimension of $A$.  
\end{lemma}

\begin{proof}
This follows from the following maps of the short exact sequences (definition of normal sheaf, \cite[Chapter II.8 Page 182]{HartshorneAG})
\begin{equation*}
\xymatrix{ 
0 \ar[r] &T_{A/k} \ar[r] & T_{A \times A/k} \otimes \O_{A} \ar[r] & \mathcal{N}_{A / A \times A} \ar[r] &0 \\
0 \ar[r] &T_{A,e} \otimes \O_A \ar[r] \ar[u]^{\cong} & (T_{A \times A, (e,e)} \otimes \O_{A \times A}) \otimes \O_{A} \ar[r] \ar[u]^{\cong} & N_{A,e} \otimes \O_A \ar[r] \ar[u]^{\cong}  &0,}
\end{equation*}
where the last isomorphism is induced by the first two. 
\end{proof}

Thus, we can compute the tangent and deformation spaces for $F_{aut}$ as $H^0(\N_A)=k^g$ and $H^1(\N_A)= k^{g^2}$.

\begin{lemma}
The  groups $Ext$ for Abelian variety $A$ in the exact sequence \ref{les} can be computed as follows in case the characteristic of the base field $k$ is greater than the dimension of $A$:
$$
Ext^1(\O_A, \O_A) = k^{2g}
$$
and
$$ 
Ext^2(\O_A, \O_A) = k^{2g^2 - g}.
$$
\end{lemma}

\begin{proof}
We will again use HKR isomorphism \cite{HKR} and hence the condition on characteristic and dimension. 
\begin{eqnarray*}
Ext^1(\O_A, \O_A) = HH^1(A) & = H^1(A, \O_A) \oplus H^0(A, \T_A) \\
& = k^g \oplus k^g 
\end{eqnarray*}
and 
\begin{eqnarray*}
Ext^2(\O_A, \O_A) = HH^2(A) & = H^2(A, \O_A) \oplus H^0(A, \T^2_A) \oplus H^1(X, \T_A) \\
& = k^{g(g-1)/2} \oplus k^{g(g-1)/2} \oplus k^{g^2}.  
\end{eqnarray*}
\end{proof}

Thus the deformation obstruction long exact sequence \ref{les} becomes
$$
0 \ra k^g \ra k^{2g} \xrightarrow{\alpha} k^g \ra k^{2g} \xrightarrow{\gamma} k^{2g^2-g} \ra \ldots.
$$
Note that now $\alpha$ does not have to be surjective to be non-zero.  Thus we cannot say that $\gamma$ is injective. So we need a geometric argument. 

\begin{theorem} \label{main1}
For any Abelian variety $A$ over an algebraically closed field $k$ of characteristic $p > 0$, any automorphism $\sigma: A \ra A$ lifts to a lift $A_K$ of $E$ over a field $K$ of characteristic zero, if and only if the Fourier-Mukai transform induced by the structure sheaf of the graph of $\sigma$, $\O_{\Gamma(\sigma)} \in D^b(A \times A)$, lifts as as an autoequivalence to $D^b(A_K) \ra D^b(A_K)$.  
\end{theorem}

\begin{proof}
Since deformation of a morphism as an autoequivalence is just the deformation of the structure sheaf of the graph as a coherent sheaf in the derived category, to prove the above statement we need to show that the support of the lifted coherent sheaf actually gives us a lift of the automorphism. This follows easily from \cite[Lemma 3.5]{tsrivast-1} and the fact that given a coherent sheaf $\F$ on a smooth projective variety, the support of the lifted coherent sheaf $\F_W$ gives a deformation of the support  of $\F$. 
\end{proof}

\begin{remark} \label{CY}
Note that the above argument also works in the case of any smooth projective variety admitting a lifting to characteristic zero. 
\end{remark}

\subsection{Are there extra lifts of automorphisms as autoequivalences?}
First, let us remark that for the case of Abelian varieties, we have the following exact sequence:

\begin{lemma} For any Artin local $W$-algebra $R$ with residue field $k$, any Abelian scheme $A$ over $Spec (R) =S$ and for any surjection $R \twoheadrightarrow R'$ of local Artin $W$-algebras such that $A' = A \otimes_R R'$,  there is an exact sequence:    
$$
0 \ra Hom_S(A, A) \ra Hom_S'(A', A').
$$
\end{lemma}

\begin{proof}
This follows from \cite[Corollary 6.2]{GIT} and 
\[
\xymatrix{
Hom_S(A,A) \ar[rd] \ar[dd] & \\
& Hom_k(A_0, A_0) \\
Hom_{S'}(A', A'). \ar[ur]}
\]
\end{proof}

This implies that $F_{aut}$ is actually a subfunctor of $F_{A_0}$, deformations of $A_0$ as a scheme, and  lift of every automorphism is unique upto lift of the base scheme.

On the other hand, one can easily see that this is not true for the deformation functor of coherent sheaves, i.e., given a fixed lift of a base Abelian scheme as a scheme, the lift of a coherent sheaf to the fixed lift will not be unique. It will be  a torsor under the $Ext^1$ group, and for Abelian variety, this group contains the group $H^1(A, \O_A)$ which is a $g$-dimensional group.  

\begin{remark} \label{subfunctor}
This is the reason that for construct algebraizable lifts of Abelian varieties, we do not just lift an line bundle but instead we choose to lift a polarization, which is actually a morphism of the Abelian variety to its dual variety. Thus using again \cite[Corollary 6.2]{GIT}, we get that the deformation functor for polarized Abelian varieties is a subfunctor of deformation functor for Abelian varieties.     
\end{remark}

Thus for Abelian varieties, the answer to the question posed in the heading is yes. There will be extra lifts of automorphism as autoequivalences. Thus, we have

\begin{proposition} \label{main2}
The  extra lifts of the automorphism as an autoequivalence are given by composition of the lifted automorphism with twist by the lift of the structure sheaf of the graph.
\end{proposition}
Note that a lift of the structure sheaf of the graph of $\sigma: A \ra A$ will be just a line bundle $\L_W$ on the graph of the lifted automorphism $\sigma_W: A_W \ra A_W$ which reduces to the structure sheaf of the graph $\O_{\Gamma(\sigma)}$. 

\subsubsection{Comparison with the deformations of the induced automorphism on product $A \times \hat{A}$} Recall from Theorem \ref{derivedAV}, that to every derived equivalence $\Phi_{\F}:D^(A) \ra D^b(A)$ one can associate an isometric automorphism $f_{\F}$ of $A \times \hat{A}$ (in case of elliptic curve $E$ an automorphism of $E \times E$). This association gives a corresponding transformation on the level of deformation functors, i.e., a natural transformation between the deformation functors $F_{coh, \F} \ra F_{aut, f_{\F}}$, where $\F$ is a (shifted) coherent sheaf \cite[Proposition 9.53]{HuyFM}. Note that the association of $\F$ with $f_{\F}$ was \textbf{not} a 1-1 correspondence, therefore the natural transformation is not injective at the level of tangent spaces of $F_{coh, \F}$ and $F_{aut, f_{\F}}$. This fits well with the discrepancy in the number of lifts  of $\F$ as an sheaf (which is not unique) and $f_{\F}$, which is unique for a chosen lift of the base Abelian variety as its deformation functor is a subfunctor for the deformation functor of Abelian varieties as schemes.



\end{document}